\documentclass[reqno,12pt]{amsart} 
\usepackage{vmargin}
\setpapersize{A4}

\usepackage{amsmath} 

     \usepackage{amsfonts}   

\usepackage{amssymb}      

\usepackage{eufrak}      

\usepackage{amscd}      

\usepackage{amsthm}      
   
\usepackage{epsfig}      

\usepackage{amstext}      

\usepackage[all,line,dvips]{xy} 
\CompileMatrices 

\newcommand{\id}{\operatorname{id}}

\newcommand{\val}{\operatorname{val}}

\newcommand{\Ro}{\operatorname{RO^+}}

   \theoremstyle{plain}
   \newtheorem{thm}{Theorem}[section]
   \newtheorem{prop}[thm]{Proposition}
   \newtheorem{lemma}[thm]{Lemma}  
   
   \theoremstyle{definition}

   \newtheorem{example}[thm]{Example}
   \theoremstyle{remark}

\usepackage{tikz,amsmath}
\usetikzlibrary{calc}

\usepackage{graphicx}

   \numberwithin{equation}{section}

\title[Exact circle maps and KMS states]{Exact circle maps and KMS states}
  
\author{Klaus Thomsen}


\email{matkt@imf.au.dk}
\address{Institut for Matematik, Aarhus University, Ny Munkegade, 8000 Aarhus C, Denmark}

\begin{document}

\maketitle


\begin{abstract} We describe the KMS-states and the ground states for the gauge action on the $C^*$-algebra of the oriented transformation groupoid of a continuous piecewise monotone and exact map of the circle.
\end{abstract}

\section{Introduction}\label{sec0} Some of the interest in the construction of $C^*$-algebras from dynamical systems stems from the role played by $C^*$-algebras and one-parameter groups of automorphisms in quantum statistical mechanics, \cite{BR}, where the algebra represents the observables and the one-parameter group the time evolution. The states of the algebra represent the states of the physical system which is being modelled, and among the states there is a distinguished class with a special relation to the one-parameter group. They represent the equilibrium states of the system, and are called KMS states, after Kubo, Martin and Schwinger who introduced the defining relation. They are associated to a real number $\beta$ which is interpreted as the inverse temperature of the physical system. 

In recent years there has been renewed focus on KMS states, partly caused by a relation to number theory established by Bost and Connes in \cite{BC}. The main purpose with the present paper is to show that there are classes of one-dimensional maps where an appropriate version of the transformation groupoid gives rise to a simple and purely infinite $C^*$-algebra for which the canonical gauge action exhibits a richness in the structure of KMS and ground states comparable to what is being discovered in systems constructed from number theory, e.g. in \cite{BC} and \cite{LR}. Specifically, it will be shown that for some of the $C^*$-algebras considered in \cite{ST} there are finitely many KMS states for all inverse temperatures above the topological entropy of the map, a unique one when the inverse temperature equals the topological entropy and infinitely many, parametrized by the state space of a finite dimensional $C^*$-algebra, at infinite inverse temperature; the so-called ground states. This occurs for continuous piecewise monotone maps $\phi$ on the circle that have at least a single turning point and are topologically exact. The number of KMS states whose inverse temperature $\beta$ exceeds the topological entropy $h(\phi)$ varies with the map, as does the structure of the ground states, depending on the orbits of the critical points. All the KMS states factor through the conditional expectation onto the copy of $C(\mathbb T)$ inside the $C^*$-algebra and there is therefore a bijective correspondence between the KMS states and measures satisfying a Radon-Nikodym relation given by the oriented transformation groupoid. When $\beta$ equals the topological entropy the measure is the pullback of Lebesgue measure under the conjugacy which turns $\phi$ into a piecewise linear map, while the measures responsible for the $\beta$-KMS states when $\beta > h(\phi)$ are purely atomic measures supported on the backward orbits of critical points that are not pre-periodic. The difference between the extremal KMS states at low temperatures and the unique $h(\phi)$-KMS state is detected by the von Neumann factors generated by the GNS-representations. For $\beta > h(\phi)$ the factors are all of type $I_{\infty}$ while the unique $h(\phi)$-KMS state gives rise to the hyperfinite $III_{\lambda}$-factor where $\lambda = e^{-h(\phi)}$.  

There are a few exceptional cases where the algebra of the oriented transformation groupoid is not simple, and has a quotient isomorphic to $C(\mathbb T)$, cf. \cite{ST}. This can occur when the degree of $\phi$ is 1 or -1, in which case there may be a non-critical fixed point $x$ for $\phi$ such that $\phi^{-1}(x) \backslash \{x\}$ only contains critical points. In these cases there are tracial states present, and they correspond to $0$-KMS states. In all other cases the $C^*$-algebra of the oriented transformation groupoid is simple, and hence has no tracial states since it is also purely infinite. The ground states arise because the fixed point algebra of the gauge action is not simple when there are critical points that are not pre-periodic; the same condition which ensures the presence of $\beta$-KMS states with $\beta > h(\phi)$. The critical points that are not pre-periodic give rise  to a finite dimensional quotient of the fixed point algebra whose state space parametrizes the ground states of the gauge action. The ground states factorises through the finite dimensional quotient of the fixed point algebra via the canonical conditional expectation onto the fixed point algebra.

This fairly rich structure of KMS states and ground states is unseen in the $C^*$-algebras coming from local homeomorphisms. In that setting it was shown in Theorem 6.8 of \cite{Th1} that even for generalized gauge actions there is an upper bound for the inverse temperatures that can occur; at least when the potential function defining the action is strictly positive or strictly negative. In contrast, there is a striking similarity, with a few intriguing differences, between the findings here and the descriptions of the KMS states for the gauge action on the $C^*$-algebras coming from the restriction of a rational map to its Julia set, both when the algebra is constructed as a Cuntz-Pimsner algebra as in \cite{IKW}, cf. Theorem 5.16 of \cite{IKW}, and as a groupoid algebra, cf. Theorem 7.5 in \cite{Th2}.

\section{KMS states and measures}\label{sec2}

Let $A$ be a $C^*$-algebra and $\alpha_t, t \in \mathbb R$, a continuous
one-parameter group of automorphisms of $A$. Let
$\beta \in \mathbb R$. A state $\omega$ of $A$ is a \emph{$\beta$-KMS state} when 
$$
\omega(a\alpha_{i\beta}(b)) = \omega(ba)
$$
for all elements $a,b$ in a dense $\alpha$-invariant $*$-algebra of $\alpha$-analytic elements,
cf. \cite{BR}. We consider in this paper the KMS states of a one-parameter group on the $C^*$-algebra constructed from a class of maps on the circle by a procedure developed in \cite{Th2} and \cite{ST}. To describe it we consider
$\mathbb T$ as an oriented space with the canonical counter-clockwise
orientation. Let $\phi : \mathbb T \to \mathbb T$ be a continuous map. There is then a
unique continuous map $f : [0,1] \to \mathbb R$ such that $f(0) \in
[0,1[$ and $\phi\left(e^{2 \pi i t} \right) = e^{2 \pi i f(t)}$
for all $t \in [0,1]$. We will refer to $f$ as \emph{the lift of
  $\phi$.} We say that $\phi$ is \emph{piecewise monotone} when there
are points $0 = c_0 < c_1 < \dots < c_N = 1$ such that $f$ is either
strictly increasing or strictly decreasing on the intervals
$]c_{i-1},c_i[, \ i = 1,2, \dots, N$. When $\phi : \mathbb T \to
\mathbb T$ is piecewise monotone and $t \in \mathbb T$ we define \emph{the $\phi$-valency} $\val(\phi,t)$ of
$t$ to be the element of the set 
$$
\left\{ (+,+), (-,-),
  (+,-), (-,+)\right\}
$$
determined by the conditions that $\val(\phi,t) = (+,+)$ when $\phi$ is strictly increasing in all
sufficiently small neighborhoods of $t$; $\val(\phi,t)
= (-,-)$ when $\phi$ is strictly decreasing in all sufficiently small
open neighborhoods of $t$; $\val(\phi,t) = (+,-)$ when $\phi$
is strictly increasing in all sufficiently small intervals to the left
of $t$ and strictly decreasing in all sufficiently small intervals to
the right of $t$; and finally $\val(\phi,t) = (-,+)$ when $\phi$
is strictly decreasing in all sufficiently small intervals to the left
of $t$ and strictly increasing in all sufficiently small intervals to
the right of $t$.

The valencies are used to define a groupoid as follows. When $x,y \in \mathbb T$ and $k \in \mathbb Z$ we write $x \overset{k}{\sim} y$ when $k =
n-m$ for some $n,m
\in \mathbb N$ such that $\phi^n(x) = \phi^m(y)$ and $\val\left(\phi^n,x\right)
= \val\left(\phi^m, y\right)$.
Set
$$
\Gamma^+_{\phi} = \left\{ (x,k,y) \in \mathbb T \times \mathbb Z \times
  \mathbb T : \  x\overset{k}{\sim} y \    \right\} .
$$ 
Then $\Gamma^+_{\phi}$ is a groupoid where the composable pairs are
$$
{\Gamma^+_{\phi}}^{(2)} = \left\{ ((x,k,y),(x',k',y')) \in
{\Gamma^+_{\phi}}^2 : \ y = x' \right\}
$$
and the product is
$$
(x,k,y)(y,k',y') = (x,k+k',y') .
$$
The inversion is given by $(x,k,y)^{-1} = (y,-k,x)$. Note that $\Gamma^+_{\phi}$ is a subgroupoid of the more standard transformation groupoid
$$
\Gamma_{\phi} = \left\{ (x,n-m,y) \in \mathbb T \times \mathbb Z \times \mathbb T : \ \phi^n(x) = \phi^m(y)\right\} .
$$

We will refer to $\Gamma^+_{\phi}$ as \emph{the oriented transformation groupoid} of $\phi$. To turn $\Gamma^+_{\phi}$ into a topological groupoid, introduce the subsets
$$
\Gamma^+_{\phi}(k,n) = \left\{ (x,l,y) \in \Gamma^+_{\phi} : \ l = k,
  \ \phi^{k+n}(x) = \phi^n(y) , \ \val\left(\phi^{k+n},x\right) = \val
  \left( \phi^n,y\right) \right\},
$$  
where $k \in \mathbb Z, \ n \in \mathbb N$, $n+k \geq 1$, $n \geq 1$. Each of these sets is the intersection of a closed and an open subset in $\mathbb T \times \mathbb Z \times \mathbb T$ and hence a locally compact Hausdorff space in the relative topology. Furthermore, $\Gamma^+_{\phi}(k,n)$ is an open subset of $\Gamma^+_{\phi}(k,n+1)$ and it follows that the union
$$
\Gamma^+_{\phi}(k) = \bigcup_{n \geq -k+1} \Gamma^+_{\phi}(k,n)
$$
is a locally compact Hausdorff space in the inductive limit topology. Since $\Gamma^+_{\phi}$ is the disjoint union of the subsets $\Gamma^+_{\phi}(k), k \in \mathbb Z$, this turns $\Gamma^+_{\phi}$ into a locally compact Hausdorff groupoid which is second countable and in fact an \'etale groupoid in the sense that the range and source maps are local homeomorphisms. See \cite{ST} for more details.

The $C^*$-algebra we consider is the (reduced) groupoid $C^*$-algebra $C^*_r\left(\Gamma^+_{\phi}\right)$ of $\Gamma^+_{\phi}$, cf. \cite{Re}. It was shown in \cite{ST} that $C^*_r\left(\Gamma^+_{\phi}\right)$ is nuclear and satisfies the universal coefficient theorem of Rosenberg and Schochet. It is purely infinite when $\phi$ is transitive, and simple if and only if $\phi$ is exact and there is no non-critical fixed point $x$ such that $\phi^{-1}(x) \backslash \{x\}$ only contains critical points, cf. Proposition 4.5 and Theorem 5.1 in \cite{ST}. Note that the unit space of $\Gamma^+_{\phi}$ is identified with $\mathbb T$ and that the range and source maps $r,s : \Gamma^+_{\phi} \to \mathbb T$ are given by $r(x,k,y) = x$ and $s(x,k,y) = y$, respectively. Since the unit space of $\Gamma^+_{\phi}$ is $\mathbb T$ the $C^*$-algebra $C(\mathbb T)$ of continuous functions on $\mathbb T$ is a canonical $C^*$-subalgebra for $C^*_r\left(\Gamma^+_{\phi}\right)$ and there is a conditional expectation 
$$
P : C^*_r\left(\Gamma^+_{\phi}\right) \to C(\mathbb T)
$$ 
obtained from the map $C_c\left(\Gamma^+_{\phi}\right) \to C(\mathbb T)$ given by restricting functions to $\mathbb T \subseteq \Gamma^+_{\phi}$, cf. \cite{Re}.

 The one-parameter action $\alpha$, called \emph{the gauge action}, whose KMS states we will examine is determined by the condition that
$$
\alpha_t(f)(x,k,y) = e^{i k t}f(x,k,y)
$$
when $f \in C_c\left(\Gamma^+_{\phi}\right)$. We will restrict the attention to the case where $\phi$ is exact and not a local homeomorphism. The last condition means that we require the presence of at least a single critical point, and there are then at least two. We shall make good use of this condition, and it seems appropriate to point out that the case where $\phi$ is exact and does not have any critical points can be handled by reference to known results. In fact, when $\phi$ is a local homeomorphism and exact it is conjugate to the algebraic homomorphism $z \mapsto z^d$ where $d \in \mathbb Z$ is the degree of $\phi$ by Theorem 4.4 in \cite{AT}. As pointed out in Lemma 3.6 of \cite{ST} the oriented transformation groupoid $\Gamma^+_{\phi}$ is isomorphic to the Renault, Deaconu, Anantharaman-Delaroche groupoid, \cite{Re}, \cite{De}, \cite{An}, of $\phi$ when $d \geq 2$ and that of $\phi^2$ when $d \leq -2$. In this way it follows from \cite{KR} that there is a unique KMS state for the gauge action on $C^*_r\left(\Gamma^+_{\phi}\right)$; it occurs at the inverse temperature $\beta = \log d$ when $d \geq2$, and $\beta = 2 \log |d|$ when $d \leq -2$. It will follow from the methods we employ here that the measure on the circle which corresponds to the KMS state is the pullback of Lebesgue measure under the homeomorphism which conjugates $\phi$ or $\phi^2$ to an algebraic homomorphism on the circle.

It may be possible to extend our methods to cases where $\phi$ is not exact, but exactness is a necessary condition for simplicity of $C^*_r\left(\Gamma^+_{\phi}\right)$ which is also sufficient when the degree of $\phi$ is not $1$ or $-1$. It seems satisfactory to know that the structure of KMS states is not related to the presence of non-trivial ideals or quotients when the $C^*$-algebra is simple. Throughout the rest of the paper we assume that $\phi$ is continuous, piecewise monotone, exact and not locally injective.

We shall use the very exhaustive and general study of KMS-states for cocycle actions on groupoid $C^*$-algebras performed by Neshveyev in \cite{N}. To this end note that the gauge action $\alpha$ is the one-parameter group of automorphisms arising from the homomorphism (or cocycle) $c_g : \Gamma^+_{\phi} \to \mathbb Z$ defined such that $c_g(x,k,y) = k$. Following \cite{Th2} we take advantage of the particular structure of the groupoid and the homomorphism $c_g$ to give the results of Neshveyev a more detailed description.

Let $W
\subseteq \Gamma^+_{\phi}$ be an open bi-section, i.e. an open subset such that $r :
W \to \mathbb T$ and $s: W \to \mathbb T$ are both injective. Then $r : W \to r(W)$ is a homeomorphism and we
denote its inverse by $r_W^{-1}$. Let $\beta
\in \mathbb R \backslash \{0\}$. As in \cite{Th2} we say that a finite Borel measure $\nu$ on $\mathbb T$ is \emph{$\left(\Gamma^+_{\phi},c_g\right)$-conformal with exponent $\beta$} when
\begin{equation}\label{Gconf}
\nu\left(s(W)\right) = \int_{r(W)} e^{\beta c_g\left(r_W^{-1}(x)\right)} \ d\nu(x)
\end{equation}
for every open bi-section $W$ of $\Gamma^+_{\phi}$. Recall that a Borel probability measure $\mu$ on $\mathbb T$ is \emph{non-atomic} when $\mu(\{x\}) = 0$ for all $x \in \mathbb T$ and \emph{purely atomic} when there is a set $A \subseteq \mathbb T$ such that $\mu(\{a\})  > 0$ for all $a \in A$ and $\sum_{a \in A} \mu(\{a\}) =1$.

\begin{prop}\label{decomp!} Let $\nu$ and $\mu$ be Borel probability
  measures on $\mathbb T$, $\nu$ non-atomic and $\mu$ purely atomic. Assume that $\nu$ and $\mu$ are both $\left(\Gamma^+_{\phi},c_g\right)$-conformal with exponent $\beta \neq 0$. Let $s \in [0,1]$.

There is then a $\beta$-KMS state $\omega$ for the gauge action on $C^*_r(\Gamma^+_{\phi})$
such that
\begin{equation}\label{decmpp}
\omega(a) = s \int_{\mathbb T} P(a) \ d\nu + (1-s) \int_{\mathbb T} P(a) \ d \mu
\end{equation}
for all $a \in C^*_r\left(\Gamma^+_{\phi}\right)$. Conversely, any $\beta$-KMS state
$\omega$ for the gauge action admits a unique decomposition of the form (\ref{decmpp}).
\end{prop}  
\begin{proof} Note that the set $\left\{x\in \mathbb T : \ \phi^k(x) = x \right\}$ is finite for each $k$; in fact, any of the intervals where $\phi^k$ is monotone contains at most one element from this set. It follows that the set of pre-periodic points is countable, and therefore in turn that the points in $\mathbb T$ with non-trivial isotropy group in $\Gamma^+_{\phi}$ is countable. Hence we can use the refinement of Neshveyev's results described in \cite{Th2}. Then the proposition follows from Theorem 2.4 in \cite{Th2} by observing that the nature of $c_g$ combined with the consistency condition in Lemma 2.3 of \cite{Th2} ensures that a purely atomic  $\left(\Gamma^+_{\phi},c_g\right)$-conformal measure must be supported on points with trivial isotropy.
\end{proof}

Thus there is a bijective correspondence between the $\beta$-KMS states and a subset of the Borel probability measures on $\mathbb T$, and our task is reduced to the identification of this subset.

Let $x\in \mathbb T$. The $\Gamma^+_{\phi}$-orbit of $x$ is the set
$$
\Ro(x) = \left\{y \in \mathbb T : \ \phi^n(x) = \phi^m(y), \ \val\left(\phi^n,x\right) = \val\left(\phi^m,y\right) \ \text{for some} \ n,m\in \mathbb N \right\},
$$
which we call \emph{the restricted orbit of $x$}. It is the subset of
the full $\phi$-orbit of $x$ restricted by the valency condition. A
purely atomic $\left(\Gamma^+_{\phi},c_g\right)$-conformal measure is
supported on at most countably many restricted orbits. It follows from
Theorem 2.4 in \cite{Th2} that a purely atomic
$\left(\Gamma^+_{\phi},c_g\right)$-conformal measure with exponent
$\beta$ defines an extremal $\beta$-KMS state if and only if it is supported on a single restricted orbit. But it is certainly not all restricted orbits which support a $\left(\Gamma^+_{\phi},c_g\right)$-conformal measure; in the terminology of \cite{Th2} they have to be $\beta$-summable for some $\beta \in \mathbb R\backslash \{0\}$. We shall determine all KMS states by proving existence and uniqueness of a non-atomic $\left(\Gamma^+_{\phi},c_g\right)$-conformal measure and by determining all $\Ro$-orbits that are $\beta$-summable for some $\beta \neq 0$.


\section{The unique non-atomic KMS state}\label{uu103}

The subset $R^+_{\phi} = \Gamma^+_{\phi}(0)$ is an open sub-groupoid of $\Gamma^+_{\phi}$ and in fact an \'etale equivalence relation in itself. The $C^*$-algebra $C^*_r\left(R^+_{\phi}\right)$ can be identified with a $C^*$-subalgebra of $C^*_r\left(\Gamma^+_{\phi}\right)$; more precisely with the fixed point algebra for the gauge action.

 When $\omega$ is a state of $C^*_r\left(\Gamma^+_{\phi}\right)$ or $C^*_r\left(R^+_{\phi}\right)$, its restriction to $C(\mathbb T)$ is given by integration with respect to a Borel probability measure, and we will say that $\omega$ is non-atomic when this measure is non-atomic and that $\omega$ is purely atomic when the measure is. In this section we combine results from \cite{S1}, \cite{S2} and \cite{ST} to prove the following

\begin{thm}\label{unique} There is exactly one non-atomic KMS-state for the gauge action on $C^*_r\left(\Gamma^+_{\phi}\right)$. The corresponding inverse temperature $\beta$ is the topological entropy of $\phi$, i.e. $\beta =h(\phi)$.
\end{thm}

Since a KMS-state restricts to a tracial state on the fixed point algebra the uniqueness part of the statement in Theorem \ref{unique} is essentially a consequence of Proposition \ref{decomp!} and the following

\begin{lemma}\label{uu106} The $C^*$-algebra $C^*_r\left(R^+_{\phi}\right)$ has at most one non-atomic tracial state.
\end{lemma}

The proof of this lemma requires some preparations. It follows from \cite{Re}, \cite{N} that a tracial state $\omega$ on  $C^*_r\left(R^+_{\phi}\right)$ factorises through the conditional expectation $P$, i.e. it is given by 
$$
\omega(a) = \int_{\mathbb T} P(a) \ d\mu
$$ 
for some probability measure $\mu$ on $\mathbb T$ which is $R^+_{\phi}$-invariant in the sense that
\begin{equation*}\label{uu107}
\mu\left(s(W)\right) = \mu\left(r(W)\right)
\end{equation*}
for every open bi-section $W \subseteq R^+_{\phi}$. To prove Lemma \ref{uu106} we must therefore show a non-atomic $R^+_{\phi}$-invariant Borel probability measure is unique. Let therefore $\mu$ be such a measure.

In the following we will write $A
 \sim B$ between two sets when $\left(A \backslash B\right) \cup \left(B \backslash A\right)$ is at most countable.

\begin{lemma}\label{uu216} Assume that $A,B$ are Borel subsets of $\mathbb T$ and $k \in \mathbb N \backslash \{0\}$ is such that $\phi^k$ is injective on both $A$ and $B$, and $\phi^k(A) \sim \phi^k(B)$. Assume also that $\val(\phi^k,x) = \val(\phi^k,y) = (+,+)$ for all $x\in A, y \in B$. It follows that $\mu(B)  = \mu(A)$.
\end{lemma}

\begin{proof} By removing countable subsets from $A$ and $B$ we can arrange that the two sets $\phi^k(A)$ and $\phi^k(B)$ agree exactly. Since $\mu$ is non-atomic we may therefore assume that this is the case. Let $\{I_i\}$ be the maximal open intervals on the circle $\mathbb T$ where $\phi^k$ is increasing. Set $A_{i,j} =  \phi^k\left(I_i\right) \cap \phi^k\left(I_j\right)$. If $A_{i,j} \neq \emptyset$ we can define an orientation preserving homeomorphism $\eta_{ij} : \phi^{-k}\left( \phi^k\left(I_j\right)\right) \cap I_i \to \phi^{-k}\left( \phi^k\left(I_i\right) \right) \cap I_j$ such that $\phi^k\left(\eta_{ij}(t)\right) = \phi^k(t)$. Then $\left\{(t,0,\eta_{ij}(t))  : t \in \phi^{-k}\left( \phi^k\left(I_j\right)\right) \cap I_i \right\}$ is an open bi-section in $R^+_{\phi} \subseteq \Gamma^+_{\phi}$. Since $\mu$ is $R^+_{\phi}$-invariant it follows that $\mu(\eta_{ij}(V)) = \mu(V)$ for every open subset $V \subseteq \phi^{-k}\left( \phi^k\left(I_j\right)\right) \cap I_i$. By regularity the same is true for all Borel subsets $V$ of $\phi^{-k}\left( \phi^k\left(I_j\right)\right) \cap I_i$. Write $\phi^k(A)$ as a disjoint union $\phi^k(A) = \sqcup X_{i,j}$ where $X_{i,j} \subseteq \phi^k\left(A \cap I_i\right) \cap \phi^k\left(B \cap I_j\right)$ are Borel sets. Then 
\begin{equation*}
\begin{split}
&\mu(B) = \sum_{i,j} \mu\left( \phi^{-k}\left(X_{i,j}\right) \cap B\right) = 
\sum_{i,j} \mu\left( \eta_{ij}\left(\phi^{-k}\left(X_{i,j}\right) \cap A\right)\right) \\
&=
 \sum_{i,j} \mu\left( \phi^{-k}\left(X_{i,j}\right) \cap A\right) = \mu\left(A\right) .
\end{split} 
 \end{equation*}

\end{proof}

To simplify notation we set $p(t) = e^{2 \pi i t}$ when $t \in \mathbb
R$. For $t \in [0,1]$, set $f_0(t) = \kappa \circ \phi\circ p(t)$
where $\kappa : \mathbb T \to [0,1[$ is the (dis-continuous) inverse of $[0,1[ \ni t
\mapsto e^{2\pi it}$. Let $\mathcal C_1$ be the critical points for
$\phi$ and choose an element $c \in \kappa\left(\mathcal C_1\right)$. Define $g : [0,2] \to [0,2]$ such that
\begin{equation*}\label{uu204}
g(t) = \begin{cases} f_0(t) & \ \text{when} \  t \in [0,1], \
  \val\left(\phi, p(t)\right) = (+,+) \\ f_0(t) +1 & \text{when} \ t
  \in [0,1], \ \val\left(\phi, p(t)\right) = (-,-) \\  f_0(t-1)+1 & \
  \text{when} \  t \in [1,2], \ \val\left(\phi, p(t)\right) = (+,+) \\
  f_0(t-1)  & \text{when} \ t \in [1,2], \ \val\left(\phi, p(t)
  \right) = (-,-) \\ c & \text{when} \ \val \left(\phi, p(t)\right)
  \in \left\{(+,-),(-,+)\right\}  \end{cases}
\end{equation*}
Then 
\begin{equation}\label{uu205}
p \circ g(t) = \phi \circ p(t), \ \ \ \ \ \ \ \ t \notin p^{-1}\left(\mathcal C_1\right).
\end{equation}

Since $\mu$ is non-atomic we can define a Borel probability measure $\tilde{\mu}$ on $[0,1]$ such that 
$$
\tilde{\mu}(A) = \mu(p(A)).
$$

\begin{lemma}\label{u1} Let $A,B$ be Borel subsets of $[0,1]$ and $k \in \mathbb N$ a natural number such that $g^k$ is injective on both $A$ and $B$, $g^k(A), g^k(B) \subseteq [0,1]$ and $g^k(A) \sim g^k(B)$. Then $\tilde{\mu}(A) = \tilde{\mu}(B)$.
\end{lemma}
\begin{proof} Note that $ \bigcup_{j=0}^k g^{-j}\left( p^{-1}(\mathcal C_1) \cup \{0,1\}\right)$ is a finite set, and let 
$$
A_0 = A \backslash \bigcup_{j=0}^k g^{-j}\left( p^{-1}(\mathcal C_1) \cup \{0,1\}\right), \ \  \   B_0 = B \backslash \bigcup_{j=0}^k g^{-j}\left( p^{-1}(\mathcal C_1) \cup \{0,1\}\right).
$$ 
Then $p$ is injective on $A_0$ and $B_0$, and it follows from (\ref{uu205}) that $\phi^k$ is injective on $p(A_0)$ and $p(B_0)$. Furthermore, $g$ is defined such that $\val\left(\phi^k,x\right) = \val\left(\phi^k,y\right) = (+,+)$ for all $x\in p(A_0), y \in p(B_0)$ because $A_0,B_0, g^k(A_0)$ and $g^k(B_0)$ are subsets of $[0,1[$. Then Lemma \ref{uu216} implies that $\tilde{\mu}(A) = \tilde{\mu}(A_0) = \tilde{\mu}(B_0) = \tilde{\mu}(B)$.
\end{proof}

Consider then the 'disconnection' $(X,\sigma)$ of the map $g : [0,2] \to [0,2]$ as introduced by F. Shultz in \cite{S1}, cf. Definition 2.1 of \cite{S1}. From Propositions 2.2 and 2.3 in \cite{S2} we cite the following facts.
\begin{enumerate}
\item[1)] $X$ is a compact metric space and $\sigma : X \to X$ is a local homeomorphism.
\item[2)] There is a continuous surjection $\pi : X \to [0,2]$ and countable subsets $I_1 \subseteq [0,2]$ and $X_1 \subseteq X$ such that $\pi|_{X \backslash X_1}$ is a conjugacy from $(X\backslash X_1,\sigma)$ onto $([0,2] \backslash I_1,g)$.
\item[3)] $X_1$ is totally $\sigma$-invariant, i.e. $\sigma^{-1}(X_1) = X_1$.
\end{enumerate}

\begin{lemma}\label{u2} $\sigma$ is exact.
\end{lemma}
\begin{proof} By Proposition 2.7 in \cite{S2} we must show that $g$ is topologically exact as defined in Definition 2.6 of \cite{S2}. Consider therefore an open non-empty subset $U$ of $[0,2]$. Then $U \cap ]0,1[ \neq \emptyset$ or $U \cap ]1,2[ \neq \emptyset$ and there is an open non-empty interval $I_0$ such that $I_0 \subseteq \overline{I_0} \subseteq U \cap ]0,1[$ or $I_0 \subseteq \overline{I_0} \subseteq U \cap ]1,2[$. Since $p(I_0)$ has non-empty interior in $\mathbb T$ there is an open non-empty interval $I \subseteq p(I_0)$. Since $\phi$ is exact there is an $N \in \mathbb N$ such that $\phi^{N-1}(I) = \mathbb T$. In particular, $I$ contains a critical point for $\phi^N$ and there are therefore open non-empty intervals $I_+,I_- \subseteq I$ such that $\val\left(\phi^N,x \right) = (\pm,\pm)$ for all $x \in I_{\pm}$ and $\phi^N\left(I_+\right) = \phi^N\left(I_-\right)$. Since $\phi$ is exact there is an $l \in \mathbb N$ such that $\phi^{N+l}\left(I_+\right) = \phi^{N+l}\left(I_-\right) = \mathbb T$. Set 
\begin{equation}\label{sets}
J_{\pm} = \left\{ x \in I_+ \cup I_- : \ \val\left(\phi^{N+l},x\right) = (\pm,\pm) \right\} .
\end{equation}
Then $\mathbb T \backslash \phi^{N+l}\left(J_+\right)$ and $\mathbb T \backslash \phi^{N+l}(J_-)$ are both finite set. Since
$p\left(g^{N+l}(I_0)\cap [0,1]\right) \supseteq \phi^{N+l}(J_+)$ when $I_0 \subseteq ]0,1[$, and $p\left(g^{N+l}(I_0)\cap [0,1]\right) \supseteq \phi^{N+l}(J_-)$ when $I_0 \subseteq ]1,2[$,
we conclude that $[0,1] \backslash g^{N+l}(I_0)$ is a finite set. Similarly, since $p\left(g^{N+l}(I_0)\cap [1,2]\right) \supseteq \phi^{N+l}(J_-)$ when $I_0 \subseteq ]0,1[$, and $p\left(g^{N+l}(I_0)\cap [1,2]\right) \supseteq \phi^{N+l}(J_+)$ when $I_0 \subseteq ]1,2[$,
we conclude that $[1,2] \backslash g^{N+l}(I_0)$ is a finite set. It follows that $[0,2] \backslash g^{N+l}(I_0)$ is a finite set. Let $\hat{g}$ be the (multivalued) map of Definition 2.6 in \cite{S1}. Then $g^{N+l}(I_0) \subseteq \hat{g}^{N+l}(\overline{I_0})$, and the latter set is closed. It follows that $\hat{g}^{N+l}(\overline{I_0}) = [0,2]$ and hence that $\hat{g}^{N+l}(U) = [0,2]$.
\end{proof}

We can now complete the \emph{proof of Lemma \ref{uu106}}: Set $X_0 = \pi^{-1}\left(]0,1[\right)$; an open non-empty subset of $X$. Since $\pi : X \backslash X_1 \to [0,2] \backslash I_1$ is a conjugacy, we can define a Borel probability measure $\hat{\mu}$ on $X_0$ such that  
$$
\hat{\mu}(B) = \tilde{\mu} \left( \pi \left(B \backslash X_1\right)\right) .
$$
 Since $X \backslash X_1$ is totally $g$-invariant and $X_1$ countable it follows from Lemma \ref{u1} that $\hat{\mu}(A) = \hat{\mu}(B)$ when $A,B$ are Borel subsets of $X_0$ such that $\sigma^k(A) = \sigma^k(B)\subseteq X_0$ and $\sigma^k$ is injective on both $A$ and $B$. Thus $\hat{\mu}$ is a Borel probability measure on $X_0$ which is invariant under the reduction $R_{\sigma}|_{X_0}$ to $X_0$ of the \'etale groupoid
$$
R_{\sigma} = \left\{ (x,y) \in X \times X : \ \sigma^n(x) = \sigma^n(y) \ \text{for some} \ n \in \mathbb N\right\} .
$$
It follows that $\hat{\mu}$ gives rise to a bounded trace on $\omega_{\mu}$ on $C^*_r\left(R_{\sigma}|_{X_0}\right)$ defined such that
$$
\omega_{\mu}(f) = \int_{X_0} f (x,x) \ d\hat{\mu}(x)
$$
when $f \in C_c\left(R_{\sigma}|_{X_0}\right)$. Now observe that $C^*_r\left(R_{\sigma}|_{X_0}\right)$ is a hereditary $C^*$-subalgebra of $C^*_r\left(R_{\sigma}\right)$; in fact, $C^*_r\left(R_{\sigma}|_{X_0}\right)$ is the closed linear span of $C_0(X_0)C^*_r\left(R_{\sigma}\right)C_0(X_0)$. Since $\sigma$ is exact by Lemma \ref{u2} it follows from Proposition 4.1 in \cite{DS} that $C^*_r\left(R_{\sigma}\right)$ is simple. It follows therefore from Lemma 4.6 in \cite{CP} that $\omega_{\mu}$ extends to a bounded trace on $C^*_r\left(R_{\sigma}\right)$. By combining Lemma \ref{u2} with Corollary 10.6 of \cite{DS} we see that $C^*_r\left(R_{\sigma}\right)$ has a unique trace state, so we conclude that if we have two non-atomic $R^+_{\phi}$-invariant Borel probability measures, $\mu_1$ and $\mu_2$, on $\mathbb T$, the two measures $\hat{\mu_1}$ and $\hat{\mu_2}$ on $X_0$ are proportional and therefore identical. This finishes the proof: For any Borel subset $B$ of $\mathbb T$ there is Borel subset $B_0$ of $X_0$ such that $p\circ \pi(B_0 \backslash X_1) \sim B$. By construction $\mu_i(B) = \hat{\mu_i}(B_0)$ and we conclude therefore that $\mu_1 = \mu_2$. $\qed$

\smallskip

We turn now to the construction of a non-atomic KMS state. Let $a > 0$. A continuous function $g : [0,1] \to \mathbb R$ is \emph{uniformly
  piecewise linear with slope $a$} when there are points $0 = c_0 < c_1 <
c_2 < \dots < c_N = 1$ such that $g$ is linear with slope $\pm a$ on
each interval $\left[c_{i-1},c_i\right], i =1,2, \dots, N$. We say that $\phi$ is \emph{uniformly
  piecewise linear with slope $a$} when its lift $f$ is. It was shown in \cite{ST} how to obtain the following conclusion from the work of F. Shultz, \cite{S2}.

\begin{thm}\label{11} There is an orientation-preserving homeomorphism
  $ k : \mathbb T \to
  \mathbb T$ such that $k \circ \phi \circ k^{-1}$ is uniformly piecewise
  linear with slope $a > 1$.
\end{thm}

We shall show below that the slope $a$ occuring in this theorem is $e^{h(\phi)}$ where $h(\phi)$ is topological entropy of $\phi$. 

Let $X$ be a compact metric space and $\sigma : X \to X$ a map which
takes Borel sets to Borel sets, and let $\nu$ a bounded Borel measure on $X$. Following Shultz, \cite{S1}, we say that $\sigma$ \emph{scales $\nu$ by a factor $a$} when $\nu\left(\sigma(E)\right) = a \nu(E)$ for every Borel subset $E$ of $X$ on which $\sigma$ is injective.

\begin{lemma}\label{uu15} There is an $a > 1$ and a non-atomic Borel probability measure $\nu$ on $\mathbb T$ which $\phi$ scales by a factor $a$. 
\end{lemma}
\begin{proof} Let $a$ and $k$ be as in Theorem \ref{11}. Being
  uniformly piecewise linear with slope $a$, the map $k \circ \phi
  \circ k^{-1}$ scales the normalized Lebesgue measure $m$ on $\mathbb T$ by a factor $a$. Set $\nu(B) = m(k(B))$. 
\end{proof}

\begin{lemma}\label{uu14} Let $\nu$ be a Borel probability measure on $\mathbb T$ and assume that $\phi$ scales $\nu$ by a factor $a > 1$. Then $\nu$ is $\left(\Gamma^+_{\phi}, c_g\right)$-conformal with exponent $\log a$.
\end{lemma}
\begin{proof} Let $W$ be an open bi-section in $\Gamma^+_{\phi}$. Then 
$$
B \mapsto \nu\left(s\left(r_W^{-1}(B)\right)\right)
$$
and 
$$
B \mapsto \int_B e^{(\log a) c_g\left(r_W^{-1}(x)\right)} \ d\nu(x)
$$ 
are both Borel measures on $r(W)$, and we want to conclude that they agree. Since 
$$
r(W) = \bigcup_{n,m} r\left(\Gamma^+_{\phi}(n,m) \cap W\right),
$$
it suffices to check that the measures agree on $r\left( W \cap \Gamma^+_{\phi}(n,m)\right)$ for each $n,m$, i.e. we may assume that $W \subseteq \Gamma^+_{\phi}(n,m)$ for some $n,m \in \mathbb N$. In fact, by definition of the topology of $\Gamma^+_{\phi}$ we may assume $W = \left\{(z,k,\eta(z)) : \ z \in U\right\}$ where $\eta : U \to V$ is an orientation preserving homeomorphism between open sets $U$ and $V$ such that $\phi^n(z) = \phi^m\left(\eta(z)\right)$ for all $z \in U$. To show that the two measures agree in this case, let $B \subseteq r(W) = U$ be a Borel subset such that $\phi^n$ is injective on $B$. Since every Borel subset of $U$ is a finite disjoint union of Borel sets for which this holds, it suffices to verify that the two measures agree on such a $B$. To this end note that $s\left(r_W^{-1}(B)\right) = \eta(B)$ and that $\phi^m$ is injective on $\eta(B)$. Since $\phi$ scales $\nu$ by the factor $a$ it follows that  
$$
\nu\left(s\left(r_W^{-1}(B)\right)\right) = \nu(\eta(B)) = a^{-m}\nu\left(\phi^m(\eta(B)\right) = a^{-m}\nu\left(\phi^n(B)\right) = a^{n-m}\nu(B) .
$$
This completes the proof since $a^{n-m}\nu(B) = \int_B e^{(\log a) c_g\left(r_W^{-1}(x)\right)} \ d \nu(x)$.
\end{proof}

In relation to the proof of Theorem \ref{unique} note that the existence of a non-atomic KMS-state follows from Proposition \ref{decomp!}, Lemma \ref{uu14} and Lemma \ref{uu15}. To determine the corresponding inverse temperature we consider again the disconnection $(X,\sigma)$ of Shultz, \cite{S1}, but now the disconnection of $([0,1],f_0)$ where $f_0 = \kappa \circ \phi \circ p$. Thus $X$ is a compact metric space, $\sigma : X \to X$ is continuous and there is a continuous surjection $\pi : X \to [0,1]$ such that $\pi \circ \sigma(z) = f_0 \circ \pi(z)$ for all $z$ in a dense subset of $X$. See \cite{S1}. We also need the observation that by construction $\pi$ is at most two-to-one everywhere, i.e. $\# \pi^{-1}(t) \leq 2$ for all $t \in [0,1]$. It follows from this that $p \circ \pi : X \to \mathbb T$ is a continuous factor map (i.e. $\phi \circ p \circ \pi = p \circ \pi \circ \sigma$) such that $\# \left(p \circ \pi\right)^{-1}(x) \leq 4$ for all $x \in \mathbb T$. This implies that $\phi$ and $\sigma$ have the same topological entropy, i.e. $h(\phi) = h(\sigma)$. It follows therefore from Proposition 2.8 and Proposition 4.3 in \cite{S2} that the value $\beta  \neq 0$ for which the gauge action on $C^*_r\left(\Gamma^+_{\phi}\right)$ has a non-atomic KMS-state is $\beta = h(\phi)$. This completes the proof of Theorem \ref{unique} because it follows from \cite{ST} that all $0$-KMS states are purely atomic when they exist.

\section{The purely atomic KMS states}

As explained in Section \ref{sec2} the extremal purely atomic $\beta$-KMS
states are supported on $\Ro$-orbits that are $\beta$-summable in the
sense of \cite{Th2}. In order for an $\Ro$-orbit $\Ro(x)$ of an element $x \in \mathbb T$ to be $\beta$-summable it must first of all be consistent, as defined in \cite{Th2}. By the nature of the groupoid $\Gamma^+_{\phi}$ and the homomorphism $c_g$ this happens if and only if the isotropy group of $x$ in $\Gamma^+_{\phi}$ is trivial, or alternatively that $x$ is not pre-periodic. When this is the case there is a well-defined map $l_x : \Ro(x) \to \mathbb R$ given by
$$
l_x(z) = e^{-k}
$$
where $k \in \mathbb Z$ is determined by the condition that $(z,k,x) \in \Gamma^+_{\phi}$, and $\Ro(x)$ is then $\beta$-summable, by definition, when
$$
\sum_{z \in \Ro(x)} l_x(z)^{\beta}  < \infty.
$$

\begin{lemma}\label{uu12} Let $\beta \neq 0$ and assume that $\Ro(x)$ is a $\beta$-summable $\Ro$-orbit. It follows that $ \Ro(x)= \Ro(c)$ where $c$ is a critical point which is not pre-periodic.
\end{lemma}
\begin{proof} Assume for a contradiction that $\Ro(x)$ does not
  contain a critical point. Then no element of $\Ro(x)$ is pre-critical and, since $\Ro(x)$ is consistent, no element of $\Ro(x)$ is pre-periodic. Since $\phi$ is exact we can find an $N \in \mathbb N$ and subsets $J_{\pm} \subseteq \mathbb T$ such that $\val\left(\phi^N,x\right) = (\pm,\pm), \ x \in J_{\pm}$, and $\mathbb T \backslash \left(\phi^N(J_+) \cap \phi^N(J_-)\right)$ is a finite set, cf. the proof of Lemma \ref{u2}. Since $x$ is not pre-periodic there is an infinite set $K \subseteq \mathbb N$ such that $\phi^k(x) \in
  \phi^N\left(J_-\right) \cap \phi^N\left(J_+\right)$ for all $k \in K$. For each $k \in K$
  there is an element $z_k \in \left(J_+ \cup J_-\right) \cap \phi^{-N}\left(\phi^k(x)\right)$ such   that $\val\left( \phi^N, z_k\right) = \val\left(\phi^k,x\right)$; i.e. $z_k \in \Ro(x)$. Since $\Ro(x)$ is $\beta$-summable it follows that
$$
\sum_{k \in K} l_x\left(z_k\right)^{\beta} = \sum_{k \in K} e^{(k-N)\beta} < \infty ,
$$
which implies that $\beta < 0$. To complete the proof we modify an
argument from the proof of Lemma 7.3 in \cite{Th2}: Let $\mathcal C_N$ be the set of
critical points for $\phi^N$. For each $k \in K,
k >N$, there are two elements in $\left(J_+ \cup J_-\right) \cap \phi^{-N}\left(\phi^k(x)\right)$; one of them is an element $z'_k \in \phi^{-N}\left(\phi^k(x)\right)
\backslash \left(\{\phi^{k-N}(x)\} \cup \mathcal C_N\right)$. Then  
$$
\left(\bigcup_{j \in \mathbb N} \phi^{-jN}(z'_k)\right) \cap
\left(\bigcup_{j \in \mathbb N} \phi^{-jN}(z'_l)\right) = \emptyset 
$$
when $k \neq l$. Since $\mathcal C_N$ is a finite set there must
therefore be a $k \in K, k
> N$, such that
$$
\left(\bigcup_{j \in \mathbb N} \phi^{-jN}(z'_k)\right) \cap \mathcal
C_N = \emptyset .
$$
Then $\val\left( \phi^{jN +N}, y\right) \in \left\{(+,+),
  (-,-)\right\}$ for all $y \in \phi^{-jN}\left(z'_k\right)$. Since
  $\phi^{-jN}\left(z'_k\right) \cap \phi^{-j'N}\left(z'_k\right) =
  \emptyset$ when $j \neq j'$ and since $\mathbb T \backslash
  \left(\phi^N(J_+)\cap \phi^N(J_-)\right)$ is a finite set, it
  follows that there is an infinite set $K'$ such that 
$$
\phi^{-jN}(z'_k) \cap \left(\phi^N(J_+) \cap \phi^N(J_-)\right) \neq \emptyset 
$$
for all $j \in K'$. For each $j \in K'$ we can choose an element $y_j \in J_+\cup J_-$ such
that $\phi^{jN+2N}(y_j) = \phi^k(x)$ and
$\val\left(\phi^{jN+2N},y_j\right) = \val\left(\phi^k,x\right)$. Then
$y_j \in \Ro(x)$, and since $\Ro(x)$ is $\beta$-summable by assumption we must have that
$$
\sum_{j \in K'} l_x\left(y_j\right) = \sum_{j\in K'} e^{(k-jN-2N)\beta} < \infty,
$$
which is impossible since $\beta < 0$. This gives us the desired
contradiction. It follows that $\Ro(x) = \Ro(c)$ for some critical point $c$ which is not pre-periodic.
\end{proof}

A critical point $c \in \mathcal C_1$ will be called \emph{terminal} when $c$ is not pre-periodic and its forward orbit $\left\{\phi^k(c) : \ k = 1,2,3, \dots \right\}$ does not contain any critical points. Let $\mathcal C_T$ be the set of terminal critical points.

\begin{lemma}\label{uu13} Let $c \in \mathbb T$ be a critical point which is not pre-periodic. Then $\Ro(c)$ is $\beta$-summable for some $\beta \neq 0$ if and only if $\beta > h(\phi)$.
\end{lemma}
\begin{proof} Let $x \in \Ro(c)$. Since $c$ is critical and $\val\left(\phi^n,x\right) = \val\left(\phi^m,c\right)$ for some $n,m$, it follows that $x$ is pre-critical. And $x$ is not pre-periodic since $c$ is not, and there is therefore a terminal critical point $c'$ containing $x$ in its backward orbit. Since $c' \in \Ro(c)$ this shows that
$$
\Ro(c) = \bigcup_{c' \in \Ro(c) \cap \mathcal C_T} \bigcup_{k =0}^{\infty} \phi^{-k}(c') .
$$  
For each $c' \in \Ro(c) \cap \mathcal C_T$ we choose $n,m \in  \mathbb N$ such that $\phi^n(c) = \phi^m(c')$, and set 
\begin{equation}\label{jan1013}
t(\beta,c') =  e^{(n-m)\beta}.
\end{equation} 
Then
$$
\sum_{z \in \Ro(c)} l_c(z)^{\beta} = \sum_{c' \in \Ro(c) \cap \mathcal C_T} t(\beta,c') \sum_{k=0}^{\infty} n_k(c')e^{-k\beta} 
$$
where $n_k(c') =  \# \phi^{-k}(c')$. Hence
$\Ro(c)$ is $\beta$-summable if and only if
\begin{equation}\label{uu500}
\sum_{k=0}^{\infty} n_k(c')e^{-k\beta} < \infty
\end{equation}
for all $c' \in \Ro(c) \cap \mathcal C_T$. 

We aim to show that for each $c' \in \Ro(c) \cap \mathcal C_T$ we have that (\ref{uu500}) holds if and only if $\beta > h(\phi)$. To estimate $n_k(c')$ observe first that after conjugation by a rotation we can arrange that $1$ is not in the backward orbit of $c'$. Since $c'$ is not periodic there is a $K \in \mathbb N$ such that $\phi^{-k}(c') \cap \mathcal C_1 = \emptyset$ for all $k \geq K$. Let $z_0 \in \phi^{-K}(c')$. We consider then again the disconnection $(X,\sigma)$ of $([0,1],f_0)$, \cite{S1}. Set $x_0 = \kappa(z_0)$ and note that $\# \phi^{-k}(z_0) =\# f_0^{-k}(x_0)$ for all $k$ since $1$ is not in the backward orbit of $c'$.  For each $x \in \bigcup_{k \in \mathbb N} f_0^{-k}(x_0)$ the pre-image $\pi^{-1}(x) \subseteq X$ of $x$ consists of the points $x^-$ and $x^+$ in $X$. We seek to compare $ \# f_0^{-k}(x_0)$ to $ \# \sigma^{-k}(x_0^+)$. For this purpose, observe that it follows by repeated application of Theorem 2.3(4)in \cite{S1} that $\pi\left(\sigma^{-k}\left(x_0^+\right)\right) \subseteq f_0^{-k}(x_0)$ for all $k$ because neither $0$ nor $1$ is in the backward orbit of $x_0$. Furthermore, we observe that when $y \in f_0^{-1}(x)$ is not a critical point for $f_0$ there is exactly one of the elements $\{y^+,y^-\}$ which is in $\sigma^{-1}(x^+)$ and the other is then in $\sigma^{-1}(x^-)$. Specifically, when $y$ is not critical and $f_0$ is increasing in a neighborhood of $y$,
$$
\sigma^{-1}(x^+) \cap \pi^{-1}(y) = \left\{y^+\right\} \ \text{and} \ \sigma^{-1}(x^-) \cap \pi^{-1}(y) = \left\{y^-\right\},
$$  
while
$$
\sigma^{-1}(x^+) \cap \pi^{-1}(y) = \left\{y^-\right\} \ \text{and} \ \sigma^{-1}(x^-) \cap \pi^{-1}(y) = \left\{y^+\right\},
$$ 
when $f_0$ is decreasing in a neighborhood of $y$, cf. \cite{S1}. This observation has bearing here because the backward orbit of $z_0$ does not contain $1$ or any critical point. It follows that all elements of $f_0^{-k}(x_0)$ are in an interval of monotonicity for $f_0$, and hence that $\# \pi^{-1}\left( f_0^{-1}(x\right) \cap \sigma^{-1}(x^{\pm}) = 1$ when $x \in f_0^{-k}(x_0)$. It follows in this way that $\pi : \sigma^{-k}(x_0^+) \to f_0^{-k}(x_0)$
is a bijection for all $k$, and we conclude that  
\begin{equation}\label{uu30}
\# \sigma^{-k}(x_0^+)  =  \# f_0^{-k}(x_0) = \# \phi^{-k}(z_0) 
\end{equation}
for all $k$. By an argument somewhat simpler than the one that proved Lemma \ref{u2} it follows that $f_0$ is exact in the sense of Definition 5.1 in \cite{S1} since $\phi$ is exact. It follows therefore from Lemma 5.2 in \cite{S1} that $\sigma$ is exact. We can then use Theorem 11 and Theorem 12 in \cite{H} to obtain constants $c > 0, d < \infty$ such that 
\begin{equation}\label{uu31}
c \leq \liminf_k \left(\lambda^{-k} \min_{x \in X} \# \sigma^{-k}(x)\right) \leq \limsup_k \left(\lambda^{-k} \sup_{x \in X} \# \sigma^{-k}(x)\right) \leq d
\end{equation}
where $\lambda = e^{h(\sigma)}$. Since $h(\sigma) = h(\phi)$ we can combine (\ref{uu30}) and (\ref{uu31}) to conclude that 
\begin{equation*}\label{uu32}
C \leq \liminf_k e^{-kh(\phi)} n_k(c') \leq \limsup_k e^{-kh(\phi)}n_k(c') \leq D.
\end{equation*}
where $C =\lambda^{-K}c$ and $D = d\lambda^{-K} \# \phi^{-K}(c')$. It follows that (\ref{uu500}) holds if and only $\beta > h(\phi)$.

\end{proof}

It follows from Lemma \ref{uu13} that every critical point $c$ which is not pre-periodic gives rise to a $\left(\Gamma^+_{\phi}, c_g\right)$-conformal measure with exponent $\beta$, and hence a $\beta$-KMS state for the gauge action when $\beta > h(\phi)$. To describe this measure note that it follows from Lemma \ref{uu13} that the sum
$$
N_{c}(\beta) = \sum_{k =0}^{\infty} \left( \# \phi^{-k}(c)\right) e^{-k\beta} 
$$
is finite when $c \in \mathcal C_T$ and $\beta > h(\phi)$. We can therefore introduce the Borel probability measure 
$$
\mu_{c,\beta} =  N_{c}(\beta)^{-1}  \sum_{k =0}^{\infty} \sum_{ v \in \phi^{-k}(c)} e^{-k\beta} \delta_v
$$
when $c \in \mathcal C_T$ and $\beta > h(\phi)$. (Here $\delta_v$ denotes the Dirac measure at $v$.) This is a $\left(\Gamma^+_{\phi}, c_g\right)$-conformal measure when $c$ is the only terminal critical point in $\Ro(c)$. To handle the general case, where $\Ro(c)$ may contain several terminal critical points, consider for each $c' \in \mathcal C_T \cap \Ro(c)$ the number $t(\beta,c')$ from (\ref{jan1013}). The sum   
$$
N_{\Ro(c)}(\beta) = \sum_{c' \in \mathcal C_T \cap \Ro(c)} t(\beta,c') \sum_{k =0}^{\infty} \left( \# \phi^{-k}(c')\right) e^{-k\beta} 
$$
is finite by Lemma \ref{uu13}. Set
\begin{equation}\label{uu51}
\mu_{\Ro(c),\beta}  = N_{\Ro(c)}(\beta)^{-1}  \sum_{c' \in \mathcal C_T \cap \Ro(c)}  t(\beta,c')\sum_{k =0}^{\infty} \sum_{ v \in \phi^{-k}(c')} e^{-k\beta} \delta_v.
\end{equation}
Then $\mu_{\Ro(c),\beta}$ is the $\left(\Gamma^+_{\phi}, c_g\right)$-conformal measure with exponent $\beta$ supported by $\Ro(c)$. Note that $\mu_{\Ro(c),\beta}$ is a convex combination of the measures $\mu_{c',\beta}, c' \in \mathcal C_T \cap \Ro(c)$. Specifically,
$$
\mu_{\Ro(c),\beta} \ = \ \sum_{c' \in \mathcal C_T \cap \Ro(c)} \alpha_{c',\beta} \mu_{c',\beta}
$$
where 
$$
\alpha_{c',\beta} = \frac{t(\beta,c')N_{c'}(\beta)}{N_{\Ro(c)}(\beta)} .
$$
We can now list all KMS states for the gauge action.

\begin{thm}\label{uu34}  Assume that $\phi : \mathbb T \to \mathbb T$ is piecewise monotone, continuous, exact and not locally injective. Let $h(\phi)$ be the topological entropy of $\phi$ and $N$ the number of $\Ro$-equivalence classes of critical points that are not pre-periodic.

There is no $\beta$-KMS state for the gauge action on $C^*_r\left(\Gamma^+_{\phi}\right)$ when $0 \neq \beta < h(\phi)$, a unique $\beta$-KMS state when $\beta = h(\phi)$ and $N$ extremal $\beta$-KMS states when $\beta > h(\phi)$. The unique $h(\phi)$-KMS state is non-atomic and all other KMS-states are purely atomic.

$0$-KMS states exist if and only $C^*_r\left(\Gamma^+_{\phi}\right)$ is not simple. This occurs only when the degree of $\phi$ is 1 or -1 and there is a non-critical fixed point $x$ for $\phi$ such that all elements of $\phi^{-1}(x) \backslash \{x\}$ are critical. Then the $0$-KMS states are in one-to-one correspondence with the Borel probability measures on $\mathbb T$.
\end{thm} 
\begin{proof} For $\beta \neq 0$ all statements follow by combining Proposition \ref{decomp!}, Theorem \ref{unique}, Lemma \ref{uu12} and Lemma \ref{uu13}. Since $0$-KMS states are trace states the last part of the theorem follows from \cite{ST}; more precisely from Theorem 5.21 and Proposition 5.16 in \cite{ST}.

\end{proof}

One consequence of this theorem is that the limits $\lim_{ \beta \downarrow h(\phi)} \mu_{\Ro(c),\beta}$ exist in the weak*-topology for all $c \in \mathcal C_T$ and that the limit is the same, namely the unique non-atomic Borel probability measure which is scaled by $\phi$. This follows by combining Theorem \ref{uu34} with Proposition 5.3.25 in \cite{BR}.

\section{The factor type of the extremal KMS states}

As is well-known a $\beta$-KMS-state $\omega$ which is extremal in the simplex of $\beta$-KMS states is a factor state, i.e. the von Neumann algebra $\pi_{\omega}\left(C^*_r\left(\Gamma^+_{\phi}\right)\right)''$ generated by the GNS-representation $\pi_{\omega}$ is a factor, cf. Theorem 5.3.30 (3) in \cite{BR}.  

\begin{prop}\label{uu301} Assume that $\omega$ is an extremal $\beta$-KMS state with $\beta > h(\phi)$. It follows that $\pi_{\omega}\left(C^*_r\left(\Gamma^+_{\phi}\right)\right)''$ is a factor of type $I_{\infty}$.
\end{prop}
\begin{proof} Let $c$ be a critical point such that the measure corresponding to $\omega$ is supported on $\Ro(c)$. Let $x\in \Ro(c)$ and let $\{f_n\}$ be a decreasing sequence of functions in $C(\mathbb T)$ converging pointwise down to the characteristic function of $\{x\}$. Then $\pi_{\omega}(f_n)$ converges in the strong operator topology to a non-zero projection $p \in \pi_{\omega}\left(C^*_r\left(\Gamma^+_{\phi}\right)\right)''$ such that $p\pi_{\omega}\left(C^*_r\left(\Gamma^+_{\phi}\right)\right)''p = \mathbb C p$. In particular, $p$ is an abelian projection in $\pi_{\omega}\left(C^*_r\left(\Gamma^+_{\phi}\right)\right)''$ which must therefore be of type $I$. It is infinite because $\Ro(c)$ contains the backward orbit of $c$ and therefore is infinite.
\end{proof}

\begin{thm}\label{uu100} The von Neumann algebra generated by the GNS representation of the unique non-atomic KMS state for the gauge action on $C^*_r\left(\Gamma^+_{\phi}\right)$ is the hyperfinite type $III_{\lambda}$-factor where $\lambda = e^{-h(\phi)}$.
\end{thm} 
\begin{proof} Let $\omega$ be the non-atomic $h(\phi)$-KMS state and let $\pi_{\omega}$ be corresponding GNS representation with cyclic vector $\Omega_{\omega}$. It was shown in \cite{ST} that $C^*_r\left(\Gamma^+_{\phi}\right)$ is nuclear and this implies that  $\pi_{\omega}\left(C^*_r\left(\Gamma^+_{\phi}\right)\right)''$ is hyperfinite. By Connes' classification of injective factors, \cite{C2}, it suffices then to show that $\pi_{\omega}\left(C^*_r\left(\Gamma^+_{\phi}\right)\right)''$ is of type $III_{e^{-h(\phi)}}$.

There is a unique $\sigma$-weakly-continuous group of automorphisms $\hat{\alpha}_t, t \in \mathbb R$, such that
$\hat{\alpha}_t\left(\pi_{\omega}(a)\right) = \pi_{\omega}\left(\alpha_t(a)\right)$,
cf. Corollary 5.3.4 in \cite{BR}. It follows then from Theorem 8.14.5 in \cite{Pe} that $\sigma_t = \hat{\alpha}_{h(\phi) t}$
is the modular group on $\pi_{\omega}(A)''$ associated to the vector state defined by $\Omega_{\omega}$. By Lemma \ref{uu106} the restriction of $\omega$ to $C^*_r\left(R^+_{\phi}\right)$ is the unique non-atomic tracial state on $C^*_r\left(R^+_{\phi}\right)$ and it is therefore an extremal tracial state. This implies that $\pi_{\omega}\left(C^*_r\left(R^+_{\phi}\right)\right)''$ is a factor. Since $\pi_{\omega}\left(C^*_r\left(R^+_{\phi}\right)\right)''$ is the fixed point algebra of $\sigma$ it follows from Proposition 2.2.2 in \cite{C1} that the Connes invariant $\Gamma\left(\pi_{\omega}\left(C^*_r\left(\Gamma^+_{\phi}\right)\right)'' \right)$ equals the Arveson spectrum $Sp(\sigma)$ of $\sigma$. It follows then from Theorem 2.3.1 in \cite{C1} that
$$
\Gamma\left(\pi_{\omega}\left(C^*_r\left(\Gamma^+_{\phi}\right)\right)'' \right)^{\perp} = \left\{ t \in \mathbb R : \ \hat{\alpha}_{h(\phi) t} = \id \right\} .
$$
Note that $(x,1,\phi(x)) \in \Gamma^+_{\phi}$ when $\val\left(\phi,x\right) = (+,+)$, and hence $C_c\left(\Gamma^+_{\phi}(1)\right) \neq 0$. Since $\hat{\alpha}_{h(\phi)t}(f) = e^{ih(\phi)t}f$ when $f \in C_c\left(\Gamma^+_{\phi}(1)\right)$, it follows that $\hat{\alpha}_{h(\phi)t} = \id \ \Leftrightarrow \ e^{ih(\phi)t} = 1  \ \Leftrightarrow \ t \in \mathbb Z\frac{2 \pi}{h(\phi)}$. Therefore $\Gamma\left(\pi_{\omega}\left(C^*_r\left(\Gamma^+_{\phi}\right)\right)'' \right) = \mathbb Z h(\phi)$, i.e. $\pi_{\omega}\left(C^*_r\left(\Gamma^+_{\phi}\right)\right)''$ is a $III_{e^{-h(\phi)}}$-factor. 
 
\end{proof}

\section{Ground states} Recall that a \emph{ground state} for the gauge action on $C^*_r\left(\Gamma^+_{\phi}\right)$ is a state $\omega$ with the property that $-i\omega\left(a^*\delta(a)\right) \geq 0$ for all $a$ in the domain of $\delta$ where $\delta$ is the generator of the gauge action, cf. \cite{BR}. Similarly, $\omega$ is a \emph{ceiling state} when $i\omega(a^*\delta(a)) \geq 0$ for all $a$ in the domain of $\delta$. 

It is not difficult to show that there are no ceiling states for the gauge action, but
ground states exist as soon as there are critical points that are not pre-periodic, and we start now to identify them. We say that a terminal critical point $c$ is \emph{final} when $c' \in \mathcal C_T, \ \phi^m(c') = \phi^n(c), \ \val\left(\phi^n,c\right) = \val\left(\phi^m,c'\right)  \Rightarrow m \geq n$. Note that the $\Ro$-orbit of a terminal critical point contains at least one final critical point. We let $\mathcal C_F$ denote the set of final critical points. 

The fixed point algebra of the gauge action is $C^*_r\left(R^+_{\phi}\right)$ and there is a conditional expectation $Q : C^*_r\left(\Gamma^+_{\phi}\right) \to C^*_r\left(R^+_{\phi}\right)$ given either by the formula
$$
Q(a) = \frac{1}{2 \pi}\int_{0}^{2 \pi} \alpha_t(a) \ dt 
$$
or as the unique extension of the map $C_c\left(\Gamma^+_{\phi}\right) \to C_c\left(R^+_{\phi}\right)$ obtained by restricting functions to $R^+_{\phi}$. Since $\mathcal C_F$ is $R^+_{\phi}$-invariant, in the sense that $(x,0,y) \in R^+_{\phi}, \ x \in \mathcal C_F \Rightarrow y \in \mathcal C_F$, there is $*$-homomorphism $\pi_F : C^*_r\left(R^+_{\phi}\right) \to C^*_r\left(R^+_{\phi}|_{\mathcal C_F}\right)$ where $R^+_{\phi}|_{\mathcal C_F}$ is the reduction of $R^+_{\phi}$ to $\mathcal C_F$, i.e.
$$
R^+_{\phi}|_{\mathcal C_F} = \left\{ (x,0,y) \in R^+_{\phi} : \ x,y \in \mathcal C_F \right\}.
$$ 
Let $\left[\mathcal C_F\right]$ denote the set of $\Ro$-equivalence classes in $\mathcal C_F$ and $[c]$ the element of $\left[\mathcal C_F\right]$ represented by an element $c\in \mathcal C_F$. Then $C^*_r\left(R^+_{\phi}|_{\mathcal C_F}\right)$ is a finite dimensional $C^*$-algebra isomorphic to
\begin{equation}\label{uu52}
\oplus_{[c] \in \left[\mathcal C_F\right]} \ M_{\#[c]}(\mathbb C).
\end{equation}

\begin{lemma}\label{uu44} $-i \pi_{F} \circ Q\left(a^*\delta(a)\right) \geq 0$ in  $C^*_r\left(R^+_{\phi}|_{\mathcal C_F}\right)$ for every $ a$ in the domain of $\delta$.
\end{lemma}
\begin{proof} When $c$ is a final critical point it follows that $\val\left(\phi^n,c\right) \neq \val\left(\phi^m,y\right)$ when $n > m$ and $\phi^n(c) = \phi^m(y)$. This implies that
$$
c \notin r \left(\bigcup_{k \geq 1} \Gamma^+_{\phi}(k) \right) = s \left(\bigcup_{k \leq -1} \Gamma^+_{\phi}(k) \right),
$$
and it follows that 
$f^*g(c,0,c') = 0$ when $c,c'$ are final critical points and either $f$ or $g$ is in $\bigcup_{k \leq -1} C_c\left( \Gamma^+_{\phi}(k)\right)$. Hence 
$\pi_{F} \circ Q(f^*g) = 0$ when $f$ or $g$ is in $\bigcup_{k \leq -1} C_c\left( \Gamma^+_{\phi}(k)\right)$. An arbitrary element $f \in C_c\left(\Gamma^+_{\phi}\right)$ can be written as a finite sum
$$
f = \sum_{k \in \mathbb Z} f_k
$$
where $f_k \in C_c\left(\Gamma^+_{\phi}(k)\right)$. We find then that
\begin{equation*}
\begin{split}
&-i \pi_F \circ Q\left(f^* \delta(f)\right) =     \sum_{k,l \in  \mathbb Z} k \pi_F \circ Q\left(f_l^*f_k\right)  =  \sum_{k  \in \mathbb Z} k \pi_F \left(f_k^*f_k\right) \\
& = \sum_{k \geq 0} k\pi_F \left(f_k^*f_k\right)   = \pi_F \circ Q \left(\left(\sum_{k \geq 0} \sqrt{k} f_k\right)^* \left(\sum_{k \geq 0} \sqrt{k} f_k\right) \right) \geq 0.
\end{split}
\end{equation*}
This proves the lemma since $C_c\left(\Gamma^+_{\phi}\right)$ is a core for $\delta$.

\end{proof}

It follows that every state $\omega$ of $C^*_r\left(R^+_{\phi}|_{\mathcal C_F}\right)$ gives rise to the ground state $a \mapsto \omega\left(\pi_F \circ Q(a)\right)$.

\begin{thm}\label{uu45} The association $\omega \mapsto \omega \circ \pi_F \circ Q$ is a bijection from the states on $C^*_r\left(R^+_{\phi}|_{\mathcal C_F}\right)$ to the ground states for the gauge action on $C^*_r\left(\Gamma^+_{\phi}\right)$.
\end{thm}
\begin{proof} Let $\omega$ be a ground state. Then $\omega$ is $\alpha$-invariant by Proposition 5.3.19 in \cite{BR} and hence $\omega = \omega \circ Q$. It remains to prove that $\omega|_{C^*_r\left(R^+_{\phi}\right)}$ factorizes through $\pi_F$, or alternatively that $\omega$ annihilates $C^*_r\left(R^+_{\phi}|_{\mathbb T \backslash \mathcal C_F}\right)$ where 
$$
R^+_{\phi}|_{\mathbb T \backslash \mathcal C_F} = \left\{ (x,0,y) \in \Gamma^+_{\phi} : \ x,y \in \mathbb T \backslash \mathcal C_F \right\}
$$
is the reduction of $R^+_{\phi}$ to $\mathbb T \backslash \mathcal C_F$. Since $C_c\left(\mathbb T \backslash \mathcal C_F \right)$ contains an approximate unit for $C^*_r\left(R^+_{\phi}|_{\mathbb T \backslash \mathcal C_F}\right)$ it suffices to show that $\omega$ annihilates $C_c\left(\mathbb T \backslash \mathcal C_F \right)$.

Let $x \in \mathbb T \backslash \mathcal C_F$. We claim that there is an element $y \in \mathbb T$ and a $k \geq 1$ such that $(x,k,y) \in \Gamma^+_{\phi}(k)$. When $x \notin \mathcal C_1 $ and $C^*_r\left(\Gamma^+_{\phi}\right)$ is simple, it follows from Lemma 6.1 in \cite{ST} that we can find such a $y$ with $k = 1$. To reach the same conclusion when $C^*_r\left(\Gamma^+_{\phi}\right)$ is not simple, note that the proof of Lemma 6.1 in \cite{ST} works as long as $\Ro(x)$ is not finite because we assume that $\phi$ is exact. It follows from Proposition 5.16 in \cite{ST} that $\Ro(x)$ is finite for all elements $x \in \mathbb T$, except possibly a non-critical fixed point. For such a fixed point $x$ we have that $(x,2,x) \in \Gamma^+_{\phi}$, and in this way we obtain the claim for all $x \in \mathbb T \backslash \mathcal C_1$. When $x \in \mathcal C_1$ is pre-periodic, say of pre-period $p$, we can take $y = x$ and $k=2p$ since $(x,2p,x) \in \Gamma^+_{\phi}$. Finally, when $x \in \mathcal C_1$ is not pre-periodic, but fail to be final because there are $n,m \in \mathbb N$ and $c \in \mathcal C_1$ such that $ n > m$ and $(x,n-m,c) \in \Gamma^+_{\phi}$ we can take $y = c$ and $k = n-m$. This proves the claim and from it follows that for each $x \in \mathbb T \backslash \mathcal C_F$ there is a $k \leq -1$ and an element $f \in C_c\left(\Gamma^+_{\phi}(k)\right)$ such that $f^*f \in C_c\left(\mathbb T \backslash \mathcal C_F \right)$ and $f^*f(z) > 0$ for all $z$ in a neighborhood of $x$. Therefore, to conclude that $\omega$ annihilates $ C_c\left(\mathbb T \backslash \mathcal C_F \right)$ and factorises through $\pi_F$ it remains only to verify that $\omega(f^*f) = 0$ when $f \in C_c\left(\Gamma^+_{\phi}(k)\right)$ for some $k \leq -1$. This follows from the ground state condition since
$$
0 \leq -i\omega(f^*\delta(f)) = k\omega(f^*f) 
$$  
implies that $\omega(f^*f) = 0$ when $k < 0$
\end{proof}

By using the description of the extremal $\left(\Gamma^+_{\phi}, c_g\right)$-conformal measure $\mu_{\Ro(c),\beta}$ given in (\ref{uu51}) it is easy to see that as $\beta$ tends to infinity the $\beta$-KMS state given by $\mu_{\Ro(c),\beta}$ will converge in the weak* topology to the ground state obtained from the extremal tracial state on the $C^*$-algebra (\ref{uu52}) which is supported on the direct summand coming from the final critical points in $\Ro(c)$.

\begin{example}\label{example} As an illustration of Theorem \ref{uu34}, let $ \alpha  > 12$ be a non-algebraic number. Let $f : [0,1] \to [0,1]$ be the uniformly piecewise linear map with slope $\alpha$ which is zero in $0,\frac{1}{4}, \frac{1}{2}, \frac{2}{3}, \frac{5}{6}$ and $1$, and increasing in $\left[0,\frac{1}{8}\right]$.

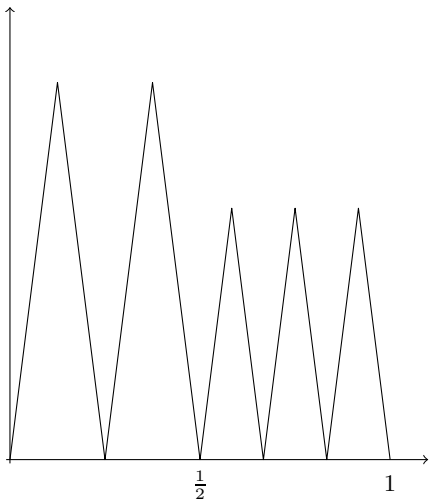
\begin{center}
\begin{tikzpicture}[x=5cm,y=5cm,font=\scriptsize]
  \draw[->] (-0.01,0) -- (1.1,0);
  \draw[->] (0,-0.01) --(0,1.2);

  \coordinate (a) at (0.5,0);
  \coordinate (b) at (0.25,0);
  \coordinate (c) at (1,0);
  \coordinate (d1) at (0.5+1/3*0.5,0);
  \coordinate (d2) at (0.5+2/3*0.5,0);
 

  \def\yyy{1}
 
  \draw[line join=round] (0,0)
  --++ (0.125,\yyy)
  -- (b)
  --++ (0.125,\yyy) 
  -- (a)
  --++ (1/3*1/2*1/2,2/3*\yyy)
  --++ (1/3*1/2*1/2,-2/3*\yyy)
  --++ (1/3*1/2*1/2,2/3*\yyy)
  --++ (1/3*1/2*1/2,-2/3*\yyy)
  --++ (1/3*1/2*1/2,2/3*\yyy)
  --++ (1/3*1/2*1/2,-2/3*\yyy)
  ;

  \node at (a) [below] {\strut$\tfrac12$};
  \node at (c) [below] {\strut$1$};

\end{tikzpicture}
\end{center}

Then $f$ is the lift of a circle map $\phi_{\alpha} : \mathbb T \to \mathbb T$ meeting the requirements of the present paper. There are $10$ critical points, $5$ of which are not pre-periodic; namely $\frac{1}{8}, \frac{3}{8}, \frac{7}{12}, \frac{3}{4}, \frac{11}{12}$ (or rather their images on the circle). The two first are in the same $\Ro$-orbit and the last three in another $\Ro$-orbit. They are all final. The topological entropy $h(\phi_{\alpha})$ is $\log \alpha$. There is a unique $\log \alpha$-KMS state and it is non-atomic. For each $\beta > \log \alpha$ there are two extremal $\beta$-KMS states, both purely atomic, corresponding to the division of the final critical points into two $\Ro$-orbit. The ground states are in one-to-one correspondence with the state space of the finite dimensional $C^*$-algebra $M_2(\mathbb C) \oplus M_3(\mathbb C)$. 

The example illustrates also how sensitive the structure of KMS states is to perturbations of the map. If we for example use $\alpha = 24$, all KMS and ground states disappear, except the non-atomic KMS state at the inverse temperature $\beta = \log 24$. 

\end{example}

\end{document}